\documentclass[noinfoline]{imsart}
\usepackage[utf8]{inputenc}
\usepackage{natbib}
\usepackage{eurosym}
\usepackage{amstext,amsthm,amsmath}
\usepackage{color}
\usepackage{amssymb,mathtools}
\usepackage{mathrsfs}
\usepackage{graphicx}
\usepackage{enumerate,paralist}
\usepackage[normalem]{ulem}
\usepackage{tikz}

\newcommand{\stirl}[2]{
  \left\{ {#1 \atop #2} \right\}
}

\newtheorem{proposition}{Proposition}[section]
\newtheorem{corollary}[proposition]{Corollary}
\newtheorem{theorem}{Theorem}
\newtheorem{lemma}[proposition]{Lemma}
\theoremstyle{definition}
\newtheorem{definition}[proposition]{Definition}
\newtheorem{remark}[proposition]{Remark}

\numberwithin{equation}{section}

\makeatletter
\renewcommand{\@fnsymbol}[1]{\@arabic{#1 }}
\makeatother

\arxiv{math.PR/0000.0000}

\begin{document}
  \title{Large-scale behavior of the partial duplication random graph}
  \runtitle{The partial duplication random graph}
  
  \begin{aug}
    \author{\fnms{Felix} \snm{Hermann}\ead[label=e1]{felix.hermann@stochastik.uni-freiburg.de}},
    \author{\fnms{Peter}
      \snm{Pfaffelhuber}\ead[label=e2]{p.p@stochastik.uni-freiburg.de}
      \ead[label=u2,url]{http://www.stochastik.uni-freiburg.de/homepages/pfaffelh/}}
  
  \runauthor{Hermann, Pfaffelhuber}

  \affiliation{University
    of Freiburg\thanksmark{m2}$\!\!$}

  \address{Abteilung f{\"u}r Mathematische Stochastik\\ Albert-Ludwigs
    University of Freiburg\\ Eckerstr. 1\\ 79104 Freiburg\\ Germany\\
    \printead{e1}\\
    \printead{e2}\\
    \printead{u2}  
  }
\end{aug}

\begin{abstract}
  \noindent
  The following random graph model was introduced for the evolution of
  protein-protein interaction networks: Let
  $\mathcal G = (G_n)_{n=n_0, n_0+1,...}$ be a sequence of random
  graphs, where $G_n = (V_n, E_n)$ is a graph with $|V_n|=n$ vertices,
  $n=n_0,n_0+1,...$ In state $G_n = (V_n, E_n)$, a vertex $v\in V_n$
  is chosen from $V_n$ uniformly at random and is partially
  duplicated. Upon such an event, a new vertex $v'\notin V_n$ is
  created and every edge $\{v,w\} \in E_n$ is copied with
  probability~$p$, i.e.\ $E_{n+1}$ has an edge $\{v',w\}$ with
  probability~$p$, independently of all other edges.

  Within this graph, we study several aspects for large~$n$. (i) The
  frequency of isolated vertices converges to~1 if
  $p\leq p^* \approx 0.567143$, the unique solution of $pe^p=1$. (ii)
  The number $C_k$ of $k$-cliques behaves like $n^{kp^{k-1}}$ in the
  sense that $n^{-kp^{k-1}}C_k$ converges against a non-trivial limit,
  if the starting graph has at least one $k$-clique. In particular,
  the average degree of a vertex (which equals the number of edges --
  or 2-cliques -- divided by the size of the graph) converges to $0$
  iff $p<0.5$ and we obtain that the transitivity ratio of the random
  graph is of the order $n^{-2p(1-p)}$. (iii) The evolution of the
  degrees of the vertices in the initial graph can be described
  explicitly. Here, we obtain the full distribution as well as
  convergence results.
\end{abstract}

\maketitle

\section{Introduction}
Random graph models are a topic of active research in probability
theory. Since the introduction of the first models, like the models of
\cite{ErdoesRenyi1959} and \cite{Gilbert1959}, several classes of
models for the evolution of networks have been introduced.
Frequently, such models try to mimic the behavior of social networks
like the internet; see \cite{CooperFrieze2003} and
\cite{BarabasiEtAl2002}. For a general introduction to random graphs
see the monographs \cite{Durrett} and \cite{Remco} and references
therein.

Another set of models aim at modeling (micro-)biological networks,
such as protein-protein interaction networks (see e.g.\
\cite{Wagner2001}, and \cite{Albert2005} for a specific application to
yeast) or metabolic networks \citep{Jeong2000}. In this paper, we
study a model introduced in \cite{Bhan2002},
\cite{PastorSatorras2003}, \cite{ChungEtAl2003} and
\cite{BebekEtAl2006}. Here, a vertex models a protein and an edge
denotes some form of interaction (e.g.\ one protein that inhibits the
expression of the second protein). Within the genome, the DNA encoding
for a protein can be duplicated (which in fact is a long evolutionary
process), such that the interactions of the copied protein are
partially inherited to the copy; see \cite{Ohno1970}. In the model we
study, every edge is copied with the same, independent, probability
$p$.

Our analysis extends previous work of \cite{ChungEtAl2003},
\cite{BebekEtAl2006} and \cite{BebekEtAl2006ip} in various
directions. We obtain results for the limit of the (expected) degree
distribution for the partial duplication model. Precisely, we are able
to determine a critical parameter $p \approx 0.567143$, the unique
solution of $pe^p=1$, below which approximately all vertices are
isolated; see Theorem~\ref{T1}. Moreover, we are able to obtain almost
sure limiting results for the number of $k$-cliques and $k$-stars in
the random graph; see Theorem~\ref{T2}. This entails precise
asymptotics of the transitivity ratio of the partial duplication
random graph; see Remark~\ref{rem:Tr}. Lastly, we study the
distribution and the large-scale behavior of the degrees of fixed
vertices; see Theorem~\ref{T3}.

\section{Model and results}

\subsection{Model}
Let us introduce some notation for (undirected) graphs. Afterwards, we
will define the random graph model we will study in the sequel.

\begin{definition}[Graph, degree, clique]\mbox{}
  \begin{asparaenum}
  \item A(n \emph{undirected}) \emph{graph} (without loops) is a tuple
    $G=(V,E)$, where $V$ is the set of vertices and $E \subseteq
    \{\{v,w\}: v,w\in V, v\neq w\}$ is the set of edges.
  \item A \emph{$k$-clique} within $G = (V,E)$ is a subset
    $V'\subseteq V$ with $|V'|=k$ and $\{\{v,w\}: v,w\in V',v\neq w\}
    \subseteq E$ (i.e.\ all vertices in $V'$ are connected). We denote
    by $C_k(G)$ the number of $k$-cliques in $G$ and by $C_k^\circ(G)
    := C_k(G)/|V|$ the relative frequency of $k$-cliques.
  \item For a graph $G=(V,E)$ and $v\in V$, we define the
    \emph{degree} of $v$ by
    $$ D_v := D_v(G):= |\{w: \{v,w\}\in E\}|.$$
    \sloppy Moreover, the absolute and relative \emph{degree
      distribution} is given by $(F_k(G))_{k=0,1,2,...}$ and
    $(F^\circ_k(G))_{k=0,1,2,...}$ through
    $$
      F_k(G) := |\{v: D_v(G)=k\}|,
      \qquad
      F_k^\circ(G) := \frac{1}{|V|} F_k(G).
    $$
    We also define their probability generating functions as
    $$
      H_q(G)
        := \sum_{k=0}^\infty F_k(G) q^k,
      \qquad
      H^\circ_q(G)
        := \sum_{k=0}^\infty F^\circ_k(G) q^k
      \qquad\text{for }q\in[0,1].
    $$
  \item A \emph{$k$-star} within $G = (V,E)$ with center $v$ is a
    vector $(v,v_1,...,v_k)$ with $v, v_1,...,v_k\in V$ and
    $\{v,v_i\}\in E, i=1,...,k$, (i.e.\ every $v_i$ is connected to
    $v$). We denote by
    $$
      S_k(G) = \sum_{\ell=k}^\infty \ell\cdots (\ell-k+1) F_\ell(G)
    $$
    the number of $k$-stars in $G$ and by $S_k^\circ(G) := S_k(G)/|V|$
    the relative frequency of $k$-stars.
  \end{asparaenum}
\end{definition}

\begin{remark}[Relationships]
  The quantities we just defined are intertwined by some
  relationships. For example, since the $S_k$-values equal
  the $k$th factorial moments of the degree distributions,
  we have
  \begin{align*}
    S_k(G) = \frac{d^k}{dq^k} H_q(G)\Big|_{q=1}.
  \end{align*}
  In particular, note that $S_1(G) = \sum_\ell \ell F_\ell(G) = 2C_2(G)$.
  This is clear, since every $1$-star counts an edge twice, having
  two possibilities of its center, while each edge corresponds to a $2$-clique.
  However, $C_k(G)$ cannot be obtained from the degree distribution, if $k\geq 3$.
\end{remark}

\noindent
We start with a basic definition of the model; see also
Figure~\ref{abb:dd}.

\begin{definition}[Partial duplication random graph\label{def:PDn}]
  Let $p \in [0,1]$. We define the following random graph process --
  called \emph{partial duplication random graph} or PDn graph --
  $\mathcal G = (G_n)_{n=n_0, n_0+1,...}$ with $G_n = (V_n, E_n)$,
  where $G_n$ is the graph at time $n=n_0, n_0+1,...$ with vertex set
  $V_n$ and (undirected) edge set $E_n \subseteq \{ \{v,w\}: v,w \in
  V_n, v\neq w\}$. Starting in some $G_{n_0} = (V_{n_0}, E_{n_0})$
  with $|V_{n_0}|=n_0$, the dynamics at time $n$ is as follows: A
  vertex $v$ is picked uniformly at random from $V_n$. Upon such an
  event, a new node $v' \notin V_{n}$ is created and every edge
  connected to~$v$ (i.e.\ every $e\in E_{n}$ with $e = \{v,w\}$ for
  some $w\in V_{n}$) is copied with probability~$p$, i.e.\ $\{v',w\}
  \in E_{n+1}$ with probability~$p$, independently of all other edges.

  We define by $C_k(n) := C_k(G_n)$ and $C_k^\circ(n) :=
  C_k^\circ(G_n)$ the number of $k$-cliques in $G_n$ and the average
  number of cliques a vertex is involved in, respectively. Similarly,
  we define by $S_k(n) := S_k(G_n)$ and $S_k^\circ(n) :=
  S_k^\circ(G_n)$ the number of $k$-stars in $G_n$ and the average
  number of $k$-stars a vertex is centered in, respectively.
  Moreover, define $F_k(n) := F_k(G_n), F^\circ_k(n) :=
  F^\circ_k(G_n), k=0,1,2,...$ the degree distribution of $G_n$ and
  its probability generating function by $H_q(n) := H_q(G_n), H_q^\circ(n) :=
  H_q^\circ(G_n)$.

  Throughout the manuscript, we will assume that the initial graph
  $G_{n_0}$ is connected and deterministic.
\end{definition}

\begin{figure}[h!]
  \begin{tikzpicture}[scale=0.75,auto=right,every node/.style={circle,fill=gray!20}]
    
    \draw[dashed,gray] plot [smooth,tension=1] coordinates {(7,4) (8.3,2) (9.8,1)};
    \draw[dashed,gray] plot [smooth,tension=1] coordinates {(7,4) (7.6,2.5) (7.9,1)};
    
    \draw plot [smooth,tension=1] coordinates {(13,4) (14.3,2) (15.8,1)};
    
    \node (n1) at (2.7,2.5) {$\;v\;\,$};
    \node (n2) at (2.7,4)   {$\phantom{v'}$};
    \node (n3) at (4.4,2.5) {$\phantom{v'}$};
    \node (n4) at (3.5,1)   {$\phantom{v'}$};
    \node (n5) at (1.9,1)   {$\phantom{v'}$};
    \node (n6) at (1,2.5)   {$\phantom{v'}$};

    \foreach \from/\to in
    {n1/n2,n1/n3,n1/n4,n1/n5,n1/n6,
      n2/n3,n4/n5}
    \draw (\from) -- (\to);
    
    \node (m1) at (8.7,2.5) {$\;v\;\,$};
    \node (m2) at (8.7,4)   {$\phantom{v'}$};
    \node (m3) at (10.4,2.5){$\phantom{v'}$};
    \node (m4) at (9.8,1)   {$\phantom{v'}$};
    \node (m5) at (7.9,1)   {$\phantom{v'}$};
    \node (m6) at (7,2.5)   {$\phantom{v'}$};
    
    \foreach \from/\to in
    {m1/m2,m1/m3,m1/m4,m1/m5,m1/m6,
      m2/m3,m4/m5}
    \draw (\from) -- (\to);
    
    \node (o1) at (14.7,2.5) {$\;v\;\,$};
    \node (o2) at (14.7,4)   {$\phantom{v'}$};
    \node (o3) at (16.4,2.5) {$\phantom{v'}$};
    \node (o4) at (15.8,1)   {$\phantom{v'}$};
    \node (o5) at (13.9,1)   {$\phantom{v'}$};
    \node (o6) at (13,2.5)   {$\phantom{v'}$};
    
    \foreach \from/\to in
    {o1/o2,o1/o3,o1/o4,o1/o5,o1/o6,
      o2/o3,o4/o5}
    \draw (\from) -- (\to);
    
    \node[circle,fill=gray!65] (v1) at (7,4)  {$v'$};
    \foreach \from/\to in
    {v1/m2,v1/m3,v1/m6}
    \draw[dashed,gray] (\from) -- (\to);
    
    \node (v2) at (13,4)  {$v'$};
    \foreach \from/\to in
    {v2/o6,v2/o3}
    \draw (\from) -- (\to);
    
    \draw[->,line width=1pt] (5.2,2.5)  -- (6.3,2.5)  node[draw=none, fill=none, midway, sloped, above] {};
    \draw[->,line width=1pt] (11.2,2.5) -- (12.3,2.5) node[draw=none, fill=none, midway, sloped, above] {};
  \end{tikzpicture}
  \caption{Illustration of one step in the PDn random graph; see also
    Definition~\ref{def:PDn}. At time $n=6$ (since there are 6
    vertices in the graph on the left), the vertex $v$ is picked
    uniformly at random. It is copied, giving rise to the new vertex
    $v'$, together with all potential edges to neighbors of $v$ (see
    the dashed lines in the middle). Then, every dashed line is kept
    independently of the others with probability $p$. The result is
    the random graph with $n=7$ vertices on the right.}
  \label{abb:dd}
\end{figure}
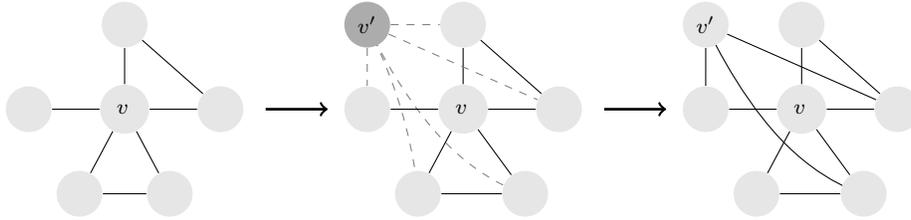

\begin{remark}[Basic observations]
  \begin{asparaenum}
  \item Since we assume that the initial graph $G_{n_0}$ is connected,
    $G_n$ consists of one connected component and singleton nodes
    which arise if a vertex is copied but none of its edges (unless
    $p=1$ where all vertices are connected), $n=n_0+1,n_0+2,...$ In
    Theorem~\ref{T1}, we will study the expected proportion of
    singleton vertices.
  \item Let $G_{n_0}$ be an $m$-partite graph for some $m\leq n_0$,
    i.e.\ there is a partition of $V_{n_0}$ into sets
    $W_1(n_0),...,W_m(n_0)$ such that $E_{n_0}\subseteq \big\{\{v,w\}:
    v\in W_i(n_0), w\in W_j(n_0) \text{ for some }i\neq
    j\}\big\}$. This means that vertices in $W_i(n_0)$ are only
    connected to vertices outside $W_i(n_0)$, $i=1,...,m$. Then, $G_n$
    is $m$-partite for all $n\geq n_0$.\\
    Indeed, if a vertex $v\in W_i(n_0)$ is copied, it is connected
    only to vertices outside $W_i(n_0)$, and so is the copied
    vertex. Iterating this argument shows that $G_n$ is $m$-partite,
    as well. Moreover, we see that the sizes
    $(W_1(n),...,W_m(n))_{n=n_0,n_0+1,...}$ of the partition elements,
    follow P\'olya's urn dynamics.
  \end{asparaenum}
\end{remark}

\begin{remark}[Related random graph models\label{rem:relG}]
  \begin{asparaenum}
  \item In \cite{PastorSatorras2003}, an extension of the PDn-model
    was introduced. After partially (with probability $p$ per edge)
    duplicating a vertex $v\in V_n$, giving rise to the new vertex
    $v'$, every vertex $w\in V_{n}$ additionally is connected to $v'$
    with probability $r/n$ for some constant $r>0$. This simple modification
    is said to induce the scale-free property (\cite{KimEtAl2002},
    \cite{BebekEtAl2006}), but, as we will see in Remark \ref{rem:connB}.3,
    this does not hold for at least some values of $p$.
  \item As stated by \cite{IspolatovEtAl2005} the famous preferential
    attachment model also shows up in a special limiting case of the
    PDn-model. Assume the case of small $p$, which implies that at
    most one edge is copied upon a duplication event, while the
    probability that at a time $n$ a fixed node $v_k$ (with degree
    $D_k(n)$) becomes connected to the new node conditioned on the
    event that at least one edge is retained equals
    $$
      \pi_{k}(n)
       := \frac{\frac{D_k(n)}n\cdot p}{\sum_{k\geq1}F_k^\circ(n)(1-(1-p)^k)}.
    $$
    Using $1-(1-p)^k\overset{p\to0}\sim pk$ and that $S_1(n)=2C_2(n)$, we
    obtain that $\pi_k(n)\xrightarrow{p\to0}D_k(n)/2C_2(n)$ for each $n$.
    So, when conditioning the PDn-model to have no isolated vertices, the
    preferential attachment model arises in the limit $p\to 0$.
  \item Another duplication model was recently introduced by
    \cite{Thoernblad} and further analyzed by \cite{BackhauszMori1, BackhauszMori2}.
    Here, the random graphs consist of disjoint cliques, almost surely. In each
    time step, a vertex $v$ is chosen uniformly at random and duplicated with
    probability $\theta$. Upon such an event, a new vertex $v'$ is created and
    connected to all neighbors of $v$ and to $v$ itself. With probability $1-\theta$,
    all edges connecting $v$ to its neighbors are deleted. For this model, the
    degree distribution was analyzed and a phase transition at $\theta=1/2$ was
    discovered in \cite{Thoernblad}. The limiting case $\theta=1/2$ and the
    maximal degree was studied in \cite{BackhauszMori1, BackhauszMori2}. The
    PDn-model is related in the case $p=\theta=1$, although the new vertex
    is not connected to the copied vertex in the PDn-model. 
  \end{asparaenum}
\end{remark}

~~
\bigskip

\subsection{Results}
Let us now come to the main conclusions about the PDn model we have
derived. First, Theorem~\ref{T1} states a critical value
$p\approx 0.567143$, below which almost all vertices have degree $0$,
i.e.\ are isolated. Its proof, which is based on a time-continuous
version of PDn and a duality argument with a piece-wise deterministic
Markov process is found in Section~\ref{S:proofT1}. Second,
Theorem~\ref{T2} studies the occurrences of certain subgraphs in the
PDn: $k$-cliques, which deliver an understanding of topological
properties of the initial graph retained during the process, and
$k$-stars, which give insights in the degree distribution, since they
describe its factorial moments. Here, we are able to obtain almost
sure limit results using Martingale theory. Third, Theorem~\ref{T3}
deals with the evolution of the degrees of fixed vertices in the
initial graph.  Here, we obtain almost sure as well as
$\mathcal L^r$-limit results.  The proofs of Theorems~\ref{T2}
and~\ref{T3} are found in Section~\ref{S:proofT23}.

\begin{remark}[Notation]
  In our Theorems, for sequences $a_1, a_2,...$ and $b_1, b_2,...$, we
  will write $a_n \stackrel{n\to\infty}\sim b_n$ iff $a_n/b_n
  \xrightarrow{n\to\infty}1$. The Gamma-function is denoted $t\mapsto
  \Gamma(t) :=\int\limits_0^\infty x^{t-1}e^{-x}dx$. Empty products,
  i.e. products of the form $\prod_{i=1}^0f(i)$, are defined to be~1.
\end{remark}

\begin{theorem}[Frequency of isolated vertices\label{T1}]
  Let $p^\ast$ be the (unique) solution of $pe^p=1$ (or $p + \log p
  =0$). Then, the following dichotomy holds:
  \begin{asparaenum}
  \item For $p\leq p^\ast$, it holds that
    $\sup_{q\in[0,1]} |H^\circ_q(n) -1| \xrightarrow{n\to\infty}0$
    almost surely. In particular, for $q=0$, we have that
    $F_0^\circ(n) \xrightarrow{n\to\infty} 1$, i.e.\ the proportion of
    isolated vertices converges to~1.
  \item For $p>p^\ast$, it holds that $\mathbf E[H^\circ_q(n)]
    \xrightarrow{n\to\infty} x_\infty < 1$ for all $q\in [0,1)$ (and
    in particular $\mathbf E[F_0^\circ(n)] \xrightarrow{n\to\infty}
    x_\infty$) with
    $$
    x_\infty := 1-\Big(1-\tfrac 1p \log\big(\tfrac 1p\big)\Big) \cdot
    \sum_{k=1}^\infty \frac{S_k^\circ(n_0)}{k!}(-1)^{k-1}
    \prod_{\ell=1}^{k-1}\Big(1-\frac{1-p^\ell}{p\ell}\Big).
    $$
  \end{asparaenum}
\end{theorem}

~

\begin{remark}[Connections to work by Bebek et al
  (2006)\label{rem:connB}]\mbox{}
  \begin{asparaenum}
  \item Previously, it has been known that $F_0^\circ(n)
    \xrightarrow{n\to\infty}1$ for $p<0.5$; see e.g.\ Lemma~2 in
    \cite{BebekEtAl2006ip}. More precisely, as explained in the same
    paper, and as a consequence of Theorem~\ref{T2} below, the
    expected number of neighbors of a randomly chosen node converges
    to 0 for $p<0.5$. Theorem~\ref{T1} extends the range for which
    $F_0^\circ(n)$ converges to~1 to the range $p\leq
    p^\ast$. Interestingly, this number already appeared in the
    analysis of PDn in a different context; see below Theorem~1 in
    \cite{ChungEtAl2003} (which we recall in Remark~\ref{rem:chung}).
  \item Consider the expected degree distribution $(\mathbf
    E[F^\circ_k(n)])_{k=0,1,2,...}$ for large $n$. It is a well-known
    consequence of a fact (usually attributed to Paul L\'evy) that a
    weak limit for such a sequence of distributions as $n\to\infty$
    exists if and only if the probability generating functions $x\mapsto \mathbf
    E[H^\circ_x(n)]$ converge to a function $h$ which is continuous at
    $x=1$. The limiting distribution then has $h$ as its probability generating
    function. As the Theorem shows, only for $p\leq p^\ast$ such a
    convergence holds and $h=1$. This implies that the degree
    distribution converges to $\delta_0$ for $p\leq p^\ast$ and there
    is no limiting degree distribution for $p>p^\ast$.

    \cite{BebekEtAl2006} call a distribution
    $(f^\circ_k)_{k=0,1,2,...}$ defective if $f^\circ_0 + f^\circ_1 +
    \cdots <1$ and non-defective if $f^\circ_0 + f^\circ_1 + \cdots =
    1$. They also raise the question of a critical value for $p$ which
    separates defective from non-defective limits of $(\mathbf
    E[F_k^\circ(n)])_{k=0,1,2,...}$.  As Theorem~\ref{T1} shows, the
    (vague) limit of $(\mathbf E[F_k^\circ(n)])_{k=0,1,2,...}$ is
    non-defective only for $p\leq p^\ast$ and defective otherwise.
    In particular, we have resolved a question raised in
    \cite{BebekEtAl2006ip}, since we have in fact shown that there is
    no limiting (probability) distribution for $(\mathbf
    E[F^\circ_k(n)])_{k=0,1,2,...}$ in the case $p>p^\ast$.
  \item \sloppy Furthermore, considering any generalization of the
    partial duplication model producing additional edges (e.g. the
    Pastor--Satorras et al model), by a suitable coupling argument and
    Theorem \ref{T1} it follows immediately that the limiting degree
    distribution is bound to be defective if the edge retaining
    probability is greater than $p^\ast$. Fueling the duplication
    mechanism with more edges, might also generate a defective limit
    for even lower values than $p^\ast$.
  \item Frequently in the literature on several random graph models,
    power laws for the (expected) degree distributions are found; see
    e.g.\ \cite{AlbertBarabasi2002}. In mathematical terms, let
    $(F_k^\circ(n))_{k=0,1,2,...}$ be the degree distribution for some
    random graph at time $n$. We say that a power-law for $n\to\infty$
    with exponent $b$ holds, if for some $c>0$,
    \begin{align*}
      \lim_{k\to\infty} k^b \lim_{n\to\infty}\mathbf E[F_k^\circ(n)] = c.
    \end{align*}
    In this sense, a power law does not exist for the PDn model since
    for all $k>0$ we have shown that $\lim_{n\to\infty} \mathbf
    E[F_k^\circ(n)] = 0$. This observation was also made by
    \cite{BebekEtAl2006ip}, who argue that the proof for the power law
    behavior of PDn given in \cite{ChungEtAl2003} is false. We come
    back to this proof in Remark~\ref{rem:chung}. Note, however, that
    for $p\leq p^\ast$ it is still possible that the connected component
    of $G_n$ satisfies a power law, i.e.\ there are $b,c>0$ with
    \begin{align*}
      \lim_{k\to\infty} k^b \lim_{n\to\infty}
      \mathbf E\Bigg[\frac{F_k(n)}{\sum_{\ell=1}^\infty F_\ell(n)}\Bigg] = c.
    \end{align*}
    Actually, the preferential attachment model arising for small $p$
    as explained in Remark~\ref{rem:relG}.2, and that model being
    known to satisfy a power-law (see e.g.\ \cite{AlbertBarabasi2002}),
    supports this conjecture. Although this power law behavior of the
    connected component has already been discussed in
    \cite{IspolatovEtAl2005}, care must be taken in order to provide a
    rigorous result. We defer the deeper analysis of the connected
    component to future research.
  \end{asparaenum}
\end{remark}

\begin{theorem}[Cliques, stars\label{T2}]
  \begin{asparaenum}
  \item Let $C_k(n_0)>0$ and $\mathcal F_\infty:=\sigma(G_n;n\geq n_0)$.
    Then, there is an $\mathcal F_\infty$-measurable random variable
    $C_k(\infty)$ with $\mathbf P(C_k(\infty)>0)>0$, such that
    \begin{align}
      \label{eq:T21}
      n^{-kp^{k-1}}C_k(n)\xrightarrow{n\to\infty} C_k(\infty),
    \end{align}
    almost surely and in $\mathcal L^2$ for each $k$. Moreover,
    \begin{align}
      \label{eq:T22}
      \mathbf E[C_k(n)] = C_k(n_0)\cdot
      \prod_{m=n_0}^{n-1}\frac{m+kp^{k-1}}{m} \stackrel{n\to\infty} \sim
      \frac{C_k(n_0)\cdot\Gamma(n_0)}{\Gamma(n_0+kp^{k-1})}n^{kp^{k-1}}.
    \end{align}
  \item For $S_k(n)$, note that $S_1(n) = 2C_2(n)$, such that the
    asymptotics for $S_1$ can be read off from the asymptotics of
    $C_2$. In addition, for each $k\geq 2$ there is an $\mathcal
    F_\infty$-measurable random variable $S_k(\infty)$, such that
    \begin{align}
      \label{eq:T21b}
      n^{-(kp+p^k)} S_k(n) & \xrightarrow{n\to\infty} S_k(\infty)
    \end{align}
    almost surely. Moreover,
    \begin{equation}
      \label{eq:T22b}
      \begin{aligned}
        \mathbf E[S_2(n)] & = \Big(S_2(n_0) + \frac{2}{p}
        S_1(n_0)\Big)\prod_{k=n_0}^{n-1}\frac{k+2p+p^2}{k} - S_1(n_0)
        \frac{2}{p} \prod_{k=n_0}^{n-1}\frac{k+2p}{k} \\ &
        \stackrel{n\to\infty} \sim \Big(S_2(n_0) + \frac{2}{p}
        S_1(n_0)\Big)\frac{\Gamma(n_0)}{\Gamma(n_0+2p+p^2)}
        n^{2p+p^2}.
      \end{aligned}
    \end{equation}
  \end{asparaenum}
\end{theorem}

\begin{remark}[Dependence on the initial graph]
  The $C_k(n)$ demonstrate a topological discrepancy between duplication
  based and preferential attachment models: In preferential attachment
  often the number of edges added to the graph in one time step is
  bounded by some constant $m$. Thus, the formation of new $(m+2)$-cliques
  is impossible, while the emergence of $k$-cliques with $k\leq m+1$
  is not -- both independent of the initial graph. In contrast to this,
  in the PDn new $k$-cliques occur if and only if there is at least one
  $k$-clique in the initial graph.
  
  Also, Theorem~\ref{T1}.2 shows that the limits of the degree distribution
  for $p>p^\ast$ very well depend on the initial values $S_k^\circ(n_0)$
  as opposed to many limiting results for preferential attachment.
\end{remark}

\begin{remark}[Moments of the degree distribution]\label{rem:moments}
  Using $\stirl nk$, the number of partitions of $\{1,...,n\}$ into
  $k$ nonempty sets, also known as the Stirling numbers of the second
  kind, we can write the moments of the degree distribution as
  \[
    M_\ell(n)
      :=\sum_{k\geq 0}k^\ell F_k^\circ(n)
      =\sum_{k\geq 0}F_k^\circ(n)\sum_{m=0}^\ell \stirl \ell m k_{\downarrow m}
      =\sum_{m=0}^\ell \stirl \ell m S_m^\circ(n),
  \]
  where $\stirl \ell\ell=1$. By \eqref{eq:T21b} we obtain
  $M_\ell(n)\sim S_\ell^\circ(n)$ almost surely and thus immediate
  limiting results for the moments.
\end{remark}

\begin{remark}[Critical values and $\mathcal L^1$-convergence]\label{rem:critval}
  \begin{asparaenum}
  \item
    It is a simple consequence of Theorem~\ref{T2}, that several
    critical values exist which distinguish cases for $C_k^\circ(n)$,
    the relative frequencies of $k$-cliques, and $S_k^\circ(n)$, the
    relative frequencies of $k$-stars, converging to~0 or
    diverging. Precisely, we obtain for $C_k^\circ$
    $$ C_k^\circ(n)\xrightarrow{n\to\infty}
      \begin{cases} \infty, & \text{ if } p>k^{-1/(k-1)}, \\
                         0, & \text{ if } p<k^{-1/(k-1)}.
      \end{cases}$$
    For the limiting case
    $p := k^{-1/(k-1)}$, we obtain 
    that $C_k^\circ(n)
    \xrightarrow{n\to\infty} C_k(\infty)$. Analogously, for $S_k^\circ$,
    and for the (unique) solution $p_k$ (in $[0,1]$) of $pk + p^k = 1$, 
    $$ S_k^\circ(n)\xrightarrow{n\to\infty}
      \begin{cases} \infty, & \text{ if } p>p_k, \\
                         0, & \text{ if } p<p_k.
      \end{cases}$$
    Surprisingly, none of those critical values equal $p^\ast$ from
    Theorem~\ref{T1}.

    There might be a connection to the $p_k$ though: Assume that the
    moments $M_x(n):=\sum k^xF_k^\circ(n)$ satisfy $M_x(n)\sim
    n^{px+p^x-1}$ not only for $x\in\mathbb N$ (as follows
    from Remark \ref{rem:moments}), but also for each $x\in\mathbb R^+$.
    For $p\geq p^\ast$, we see that the inequality $px+p^x\geq
    px+e^{-px}>1$ holds for all $x>0$, while for $p < p^\ast$ there
    are $x>0$ with $px+p^x < 1$. Thus, $p^\ast$ is the smallest value
    of $p\in[0,1]$ such that there is no positive solution $x$ of
    $px+p^x=1$. Hence, all $M_x(n)$ tend to $\infty$ if and only if
    $p>p^\ast$.
  \item Unfortunately, we are not able to show that the convergence in
    \eqref{eq:T21b} also holds in $\mathcal L^1(\mathcal F_\infty)$ if
    $k\geq 2$ and thus we cannot rule out the possibility that
    $S_k(\infty)$ is trivial, i.e.\ we cannot rule out
    $S_k(\infty)=0$. It should be possible to use a technique similar
    to the proof of the $\mathcal L^2$ convergence of the $C_k(n)$,
    but it is much more difficult since additionally there are three
    possible relations of the centers $c_1,c_2$ of two $k$-stars $s_1$
    and $s_2$:
    \[
      a)\ c_1=c_2,\qquad
      b)\ (c_1,c_2)\in E_n,\qquad
      c)\ \text{neither $a)$ nor $b)$},
    \]
    where each of those as well as the number of shared nodes or edges
    influence the evolutions of $k$-star pairs in a different way.
    Since the structures of pairs of $k$-stars are so complex, another
    approach might be more suitable in order to show non-triviality of
    $S_k(\infty)$.
  \end{asparaenum}
\end{remark}

\begin{remark}[Transitivity ratio\label{rem:Tr}]
  The transitivity ratio $Tr(G)$ of a graph $G = (V,E)$ is defined
  via $C_3(G)$ and $S_2(G)$ by
  $$ Tr(G) := \frac{6C_3(G)}{S_2(G)}.$$
  (Precisely, it is defined by the quotient of three times the number
  of triangles $C_3(G)$ and the number of connected triples, i.e.\ the
  number of triples $v,w,u\in V$ with $\{v,w\}, \{w,u\} \in E$. Each
  connected triple is counted twice by $S_2(G)$ upon summing over
  vertex $w$.) Hence, we find that 
  \begin{align*}
    \frac{\mathbf E[6C_3(n)]}{\mathbf E[S_2(n)]}
    \stackrel{n\to\infty}\sim \frac{6C_3(n_0)}{S_2(n_0) + \tfrac 2p
      S_1(n_0)} \frac{\Gamma(n_0 + 2p+p^2)}{\Gamma(n_0+3p^2)}
    n^{-2p(1-p)}.
  \end{align*}
  Moreover, $n^{2p(1-p)}Tr(n)$ converges (at least on the set
  $S_2(\infty)\neq 0$) to some integrable random variable by
  Theorem~\ref{T2}.
\end{remark}

\begin{theorem}[Degree evolution of the initial vertices\label{T3}]
  Let $V_{n_0} = \{1,...,n_0\}$, i.e.\ we number the initial vertices
  by $1,...,n_0$. In addition, let $D_k(n)>0$ be the degree of vertex
  $k\leq n_0$ at time $n$. Then, for $n\geq n_0$ and $\ell\geq a$,
  \begin{align}\label{eq:T32}
    \mathbf{P}(D_k(n)=\ell|D_k(n_0)=a)
     &= \sum_{m=a}^\ell(-1)^{m-a}{\ell-1 \choose m-1}{m-1\choose a-1}
         \prod_{j=n_0}^{n-1}\left(1-\frac{pm}{j}\right).
  \end{align}
  Moreover, there is an almost surely positive random variable $D_k(\infty)$,
  such that
  \begin{align}
    \label{eq:T33}
    n^{-p}D_k(n)\xrightarrow{n\to\infty} D_k(\infty)
  \end{align}
  almost surely and in $\mathcal L^r$ for each $r\geq 1$ and, using
  $\ell_{\uparrow m} := \ell \cdots (\ell+m-1)$ for $m=1,2,\ldots$,
  \begin{align}
    \label{eq:T34}
    \mathbf E[n^{-mp}D_k(n)_{\uparrow m}]
     &= \frac{D_k(n_0)_{\uparrow m}}{n^{mp}}\prod_{\ell=n_0}^{n-1}\tfrac{\ell+mp}\ell
      \xrightarrow{n\to\infty}\frac{D_k(n_0)_{\uparrow m}\cdot\Gamma(n_0)}{\Gamma(n_0+mp)}
      = \mathbf E[D_k(\infty)^m].
  \end{align}
\end{theorem}

\begin{remark}[Degree evolution of arbitrary vertices]
  In the case of $k>n_0$ we can obtain results for the behavior
  of $D_k(n)$ by conditioning on the graph at time $k$, i.e. considering
  $G_k$ as initial graph.
\end{remark}

\begin{remark}[Connection to the P\'olya urn]
  In the special case $p=1$, the sequence $(D_k(n))_{n=n_0,
    n_0+1,...}$ is connected to P\'olya's urn. Note that the degree of
  vertex $v_k$ increases by one at time $n$ iff one of the neighbors of
  $v_k$ is copied. In P\'olya's urn, start with $n_0$ balls,
  where~$D_k(n_0)=a$ balls are red and all others are black. Then, as
  usual, pick a ball from the urn at random, and put it back together
  with a second ball of the same color. From this construction,
  $D_k(n)$ is equal in distribution to the number of red balls when
  there are $n$ balls in the urn for all $n$.
  Of course, it is well-known that in this case, the probability that
  there are $\ell$ red balls in the urn at the time when there are $n$
  balls in total equals
  \begin{align}\label{eq:T32b}
    \binom{n-n_0}{\ell-a}
       \frac{a_{\uparrow(\ell-a)}\cdot(n_0-a)_{\uparrow(n-n_0-\ell+a)}}
            {n_{0\uparrow (n-n_0)}},
  \end{align}
  with $m_{\uparrow k}:=m\cdot(m+1)\cdots(m+k-1)$
  (e.g. (4.2) in \cite{JohnsonKotz1977}).\\
  Using (5) of Chapter 1 of \cite{Riordan1968} we obtain
  \begin{align*}
    \binom{n-n_0}{\ell-a}
     &=\binom{(n-a-1)-(n_0-a-1)}{\ell-a}\\
     &=\sum_{m=0}^{\ell-a}(-1)^m\binom{n-a-1-m}{\ell-a-m}\binom{n_0-a-1}{m}.
  \end{align*}
  From this it follows, that~\eqref{eq:T32b} equals the
  right hand side of~\eqref{eq:T32} if $p=1$.

  \noindent
  Also section 6.3.3 in \cite{JohnsonKotz1977} shows that the
  proportion of red balls in the urn converges to a
  $\beta$-distributed random variable with parameters $a$ and
  $n_0-a$. Therefore it is not surprising that for $p=1$ the moments
  of $D_k(\infty)$ as given in \eqref{eq:T34} match those of the
  $\beta(a,n_0-a)$-distribution.

  \noindent
  The connection of \eqref{eq:T32} to an extension of P\'olya's urn
  would look as follows: Consider an urn, starting with $n_0$ balls,
  $a$ of which are red and $n_0-a$ of which are black. In each step,
  choose a ball at random from the urn. If the ball is black, put it
  back to the urn together with another black ball. If the ball is
  red, put it back to the urn together with another ball. The color of
  the additional ball is red with probability $p$ and black with
  probability $1-p$. Then, the chance that there are $\ell$ red balls
  in the urn at the time when there are a total of $n$ balls in the
  urn equals the right hand side of \eqref{eq:T32}.
\end{remark}

\begin{remark}[The limiting distribution]
  The limiting distribution with moments given by the right
  hand side of~\eqref{eq:T34} seems to be not well--known. Thus far,
  we deduced, using Stirling's formula, that for $p<1$
  \begin{align*}
    \|D_k(\infty)\|_{\mathcal L^\infty}
     &= \lim_{m\to\infty}\|D_k(\infty)\|_{\mathcal L^m}
      = \lim_{m\to\infty}\Bigg(\frac{\Gamma(D_k(n_0)+m)}{\Gamma(n_0+pm)}\Bigg)^{1/m}
      = \infty,
  \end{align*}
  which shows that (in contrast to the special case $p=1$) $D_k(\infty)$
  is not bounded. Also, Stirling's formula shows that the moments satisfy
  \begin{align*}
    \sum_{m=1}^\infty\mathbf E[D_k(\infty)^m]^{-1/(2m)}=\infty,
  \end{align*}
  which is Carleman's condition for the determinacy of the corresponding
  Stieltjes moment problem (e.g. \cite{ShohatTamarkin1943}, Theorem 1.11).
  Thus, the limiting distribution is defined by its moments.
\end{remark}

\section{Preparation}

\subsection{Some recursions}
We collect some simple calculations in this section. Throughout, we
denote by $(\mathcal F_n)_{n=n_0, n_0+1,...}$ the filtration generated
by $\mathcal G = (G_n)_{n=n_0,n_0+1,...}$.

\begin{proposition}[Evolution of $F, H, S, C$\label{P:FHSC}]
  It holds that
  \begin{align}\label{eq:P11}
    \mathbf E[F_k(n+1)|\mathcal F_n] & = F_k(n) + p(k-1)
    F^\circ_{k-1}(n) - pk F^\circ_k(n) \\ & \qquad \qquad \qquad
    \qquad \notag + \sum_{\ell\geq k} F^\circ_\ell(n)
    \binom{\ell}{k} p^k(1-p)^{\ell-k},\\
    \label{eq:P12}    
    \mathbf E[H_q(n+1)|\mathcal F_n] & = H_q(n) -pq(1-q) \frac{d}{ds} H^\circ_s(n)\Big|_{s=q} + H^\circ_{1-p+pq}(n),\\
    \label{eq:P13} 
    \mathbf E[S_k(n+1)|\mathcal F_n] & = \Big(1+\frac{pk+p^k}{n}\Big)S_k(n) + \frac{pk(k-1)}{n}S_{k-1}(n),\\
    \label{eq:P14}
    \mathbf E[C_k(n+1)|\mathcal F_n] & = C_k(n) \Big(1 + \frac{k}{n}
    p^{k-1}\Big).
  \end{align}
\end{proposition}

\begin{proof}
  Let us start with \eqref{eq:P11}. The quantity $F_k$ increases in
  two cases: either, a vertex of degree $\ell\geq k$ is copied,
  together with $k$ edges (which has probability $\binom\ell k
  p^k(1-p)^{\ell-k}$), or one of the neighbors of a vertex of degree
  $k-1$ is copied together with the connecting edge. On the other
  hand, $F_k$ decreases by one, if one of the neighbors of a vertex of
  degree $k$ is copied together with the connecting edge. These three
  cases make up the right hand side of~\eqref{eq:P11}.

  \noindent
  For \eqref{eq:P12}, recall the definition of $H_q$. We multiply
  \eqref{eq:P11} by $q^k$ and sum in order to obtain
  \begin{align*}
    \mathbf E[H_q(n+1)&-H_q(n)|\mathcal F_n] = pq^2\sum_{k=1}^\infty
        (k-1)F^\circ_{k-1}(n)q^{k-2} - pq\sum_{k=0}^\infty
        kF^\circ_{k}(n)q^{k-1} \\ & \qquad \qquad \qquad \qquad +
        \sum_{k=0}^\infty\sum_{\ell=k}^\infty
           F^\circ_\ell(n) \binom\ell k p^k(1-p)^{\ell-k} q^k \\
    & = -pq(1-q) \frac{d}{ds}\Big(\sum_{k=0}^\infty F^\circ_k(n) s^k\Big)\Big|_{s=q} +
        \sum_{\ell=0}^\infty F^\circ_\ell(n) (1-p+pq)^\ell \\
    & = -pq(1-q)
        \frac{d}{ds} H^\circ_s(n)\Big|_{s=q} + H^\circ_{1-p+pq}(n).
  \end{align*}
  We now turn to \eqref{eq:P13}. Again, use \eqref{eq:P11}, multiply
  by $k \cdots (k-m+1) =: k_{\downarrow m}$ and sum for
  \begin{align*}
    \mathbf E & [S_m(n+1)-S_m(n)|\mathcal F_n] \\
    & = \sum_{k=m}^\infty\Bigg( pk_{\downarrow m}(k-1)
    F^\circ_{k-1}(n) \\ & \qquad \qquad \qquad \qquad \qquad -
    pk_{\downarrow m} k F^\circ_k(n) + \sum_{\ell=k}^\infty
    F^\circ_\ell(n) k_{\downarrow m} \binom{\ell}{k}
    p^k(1-p)^{\ell-k}\Bigg) \\
    & = p\sum_{k=m}^\infty ((k+1)_{\downarrow m} - k_{\downarrow m})
    kF^\circ_k(n) \\ & \qquad \qquad \qquad \qquad \qquad +
    p^m\sum_{\ell=m}^\infty\sum_{k=m}^ \ell F^\circ_\ell(n)
    \ell_{\downarrow m} \binom{\ell-m}{k-m}
    p^{k-m}(1-p)^{\ell-k}\\
    & = p \sum_{k} \big(mk_{\downarrow m} + m(m-1)k_{\downarrow(m-1)})
    \big)F^\circ_k(n) + p^m \sum_{\ell} F^\circ_\ell(n)
    \ell_{\downarrow m} \\
    & = \frac{pm + p^m}n S_m(n) + \frac{pm(m-1)}{n} S_{m-1}(n),
  \end{align*}
  where we have used that
  \begin{align*}
    k\cdot\big((k&+1)_{\downarrow m}-k_{\downarrow m}\big)\\
    &= \big(k-m+1+(m-1)\big)\cdot k_{\downarrow (m-1)}\big(k+1-(k-m+1)\big)\\
    &= mk_{\downarrow m} + m(m-1)k_{\downarrow (m-1)}.
  \end{align*}
  For \eqref{eq:P14}, a $k$-clique arises if a vertex $v$ which is
  member of a $k$-clique is copied, together with all $k-1$ edges
  connecting $v$ to the other members of the clique. Hence,
  \begin{align*}
    \mathbf E[C_k(n+1)|\mathcal F_n] & = C_k(n) \Big(1 + \frac{k}{n}
    p^{k-1}\Big).
  \end{align*}
\end{proof}

\begin{remark}[Scale-free property\label{rem:chung}]
  \begin{asparaenum}
    \item
      In \cite{ChungEtAl2003}, the authors show the following: If for some
      $b>0$ (necessarily we will have $b>1$) and $c>0$ it holds that
      \begin{align}
        \label{eq:chung}
        \lim_{k\to\infty} k^b \lim_{n\to\infty} \mathbf E[F^\circ_k(n)] =
        c,
      \end{align}
      then, $b$ must satisfy $p(b-1) = 1-p^{b-1}$.

      Let us briefly recall the arguments leading to this power-law
      behavior of the (expected) degree distribution. Starting off
      with~\eqref{eq:P11}, taking expectations on both sides, and setting
      $\mathbf E[F_k(n)] = ck^{-b}n + o(n)$ for some $c,b$, we see that
      for $n\to\infty$, if a stationary state for $\mathbf E[F^\circ_k]$
      is reached,
      \begin{align*}
        ck^{-b} & \stackrel{k\to\infty}\sim p(k-1) c(k-1)^{-b} - pk
        ck^{-b} + c\sum_{\ell\geq k} \ell^{-b} \binom{\ell}k
        p^k(1-p)^{\ell-k}.
      \end{align*}
      Note that $k(1-1/k)^{-b} -k \stackrel{k\to\infty}\sim b$, and (see
      Lemma~2 in \cite{ChungEtAl2003})
      \begin{align*}
        \sum_{\ell\geq k} \ell^{-b} \binom{\ell}k p^k(1-p)^{\ell-k}
        \stackrel{k\to\infty} \sim k^{-b}p^{b-1}.
      \end{align*}
      Therefore, by dividing by $ck^{-b}$, the parameter $b$ must satisfy
      \begin{align}\label{eq:chungRelation}
        1 & = pb - p + p^{b-1},
      \end{align}
      the desired relationship.

  ~ 

      \noindent
      Of course, with this proof \cite{ChungEtAl2003} only show an
      assertion about the scaling exponent $b$ in the case that the
      limiting distribution of $(\mathbf
      E[F_k^\circ(n)])_{k=0,1,2,...}$ satisfies a power law. No
      assertion is made if such a power law exists.  In order to
      resolve this, consider $p\geq p^\ast$ as given in
      Theorem~\ref{T1}, that is $p\geq e^{-p}$.  For such $p$ the
      inequality $px+p^x\geq px+e^{-px}>1$ holds for all $x>0$.  More
      precisely, the desired relationship has a solution $b>1$ if and
      only if $p < p^\ast$. However, in this case we have seen that
      $(\mathbf E[F_k^\circ(n)])\xrightarrow{n\to\infty} \delta_{k0}$
      and so \eqref{eq:chung} cannot hold. We conclude that no
      power--law behavior is possible.
    \item The works on the Pastor--Satorras et al modification we
      mentioned in Remark \ref{rem:relG} claim the same power law
      \eqref{eq:chungRelation} to hold, again with $p(b-1)=1-p^{b-1}$,
      irrespective of $r$ as long as $r>0$ (see Theorem~1 in
      \cite{BebekEtAl2006ip} using arguments similar to those of
      \cite{ChungEtAl2003} and (18) in \cite{KimEtAl2002}, where the
      connection gets clear by multiplying $1-\delta$ and substituting
      $p=1-\delta$). It is also recognized that such a power
      law can only exist for $p\leq p^\ast$. However, the proofs of
      this power law only show stationarity of a degree distribution
      with power law. The question of convergence in the
      sense of~\eqref{eq:chung} still is an open problem.
    \end{asparaenum}
\end{remark}

\subsection{An auxiliary process}
\label{ss:mathcalX}
In the proof of Theorem~\ref{T1}, we will need a piece-wise
deterministic process which we introduce here. There, we will obtain
and use a \emph{duality} (see Subsection \ref{subs:duality}), i.e. a
relationship of the form
$$
  \mathbf E[H_{1-x}(t)] = \mathbf E[H_{1-X_t}(0)|X_0=x]
$$
for continuous-time versions of the probability generating functions
of the degree distributions $H$ and a $[0,1]$-valued process
$\mathcal X = (X_t)_{t\geq0}$, which jumps from $x$ to $px$ at rate~1
and in between jumps follows the logistic equation $\dot X = pX(1-X)$.
Recall that such piece-wise deterministic processes have been studied
recently in more detail; see e.g.\ \cite{Davis1984}, \cite{CostaDufour2008},
\cite{AzaisBardetGenadotKrellZitt2014}.

\begin{lemma}[The auxiliary process $\mathcal X$\label{l:X}]
  Let $p\in[0,1]$ and $\mathcal X = (X_t)_{t\geq 0}$ be a Markov
  process with state space $[0,1]$ and generator
  \begin{align}\label{eq:genX}
    G_{\mathcal X}f(x) = px(1-x)f'(x) + (f(px) - f(x))
  \end{align}
  for $f\in\mathcal C^1_b([0,1])$ and $X_0 \in [0,1]$. In addition,
  let $p^\ast \approx 0.567143$ be the unique solution of $pe^p=1$ (or
  $p + \log p = 0$).

  Then, if $p\leq p^\ast$, it holds that $X_t \xrightarrow{t\to\infty}
  0$ almost surely, whereas if $p>p^\ast$ it holds that $\mathcal X$
  is ergodic and $X_t \xRightarrow{t\to\infty} X_\infty$ for some
  $[0,1]$-valued random variable $X_\infty$ with $\mathbf
  P(X_\infty>0)=1$ and
  $$\mathbf E[X_\infty^k] = \Big(1 - \tfrac 1p \log\big(\tfrac
  1p\big)\Big) \cdot \prod_{\ell=1}^{k-1} \Big( 1 -
  \frac{1-p^\ell}{p\ell}\Big).$$
\end{lemma}

\begin{proof}
  We consider the process $-\log\mathcal X = (-\log X_t)_{t\geq 0}$
  with state space $[0,\infty)$. From \eqref{eq:genX}, we read off
  that this process has the generator
  $$ G_{-\log \mathcal X}g(y) = -p(1-e^{-y}) g'(y) + g(y + \log(1/p)) - g(y).$$
  In other words, $-\log\mathcal X$ decreases at rate $p(1-e^{-y})$
  at time $t$ if $-\log X_t$ equals $y$ and increases by $\log(1/p)$
  at the times of a Poisson process. Note that $X_t \xrightarrow{t\to\infty} 0$
  iff $-\log X_t \xrightarrow{t\to\infty} \infty$.

  We start with the case $p<p^\ast$. Here, we can couple the process
  $-\log \mathcal X$ with a process $\mathcal U = (U_t)_{t\geq 0}$
  with generator 
  $$ G_{\mathcal U} g(y) =  -p g'(y) + g(y + \log(1/p)) - g(y)$$
  by using the same Poisson processes for $-\log\mathcal X$ and
  $\mathcal U$. Since $1-e^{-y} \leq 1$, we have that $U_t\leq -\log
  X_t$. However, we can write $\mathcal U$ as
  $$ U_t = U_0 - pt + \log(1/p) P_t$$
  for some unit-rate Poisson process $\mathcal P = (P_t)_{t\geq 0}$
  and by the law of large numbers for Poisson processes (i.e.\
  $\tfrac{P_t}{t} \xrightarrow{t\to\infty} 1$ almost surely), we see
  that $U_t\xrightarrow{t\to\infty} \infty$ almost surely, if
  $\log(1/p)>p$ or $p<p^\ast$. Since $U_t\leq -\log X_t$, this implies
  $-\log X_t\xrightarrow{t\to\infty} \infty$ or
  $X_t\xrightarrow{t\to\infty} 0$, as claimed.

  Now, we turn to the case $p>p^\ast$. First, we have to prove
  ergodicity of $-\log\mathcal X$ (which is equal to ergodicity of
  $\mathcal X$). Let $T_z := T_z^{-\log\mathcal X} := \inf\{t\geq 0:
  -\log X_t = z\}$. According to \cite{Davis1983}, Theorem 3.10, we
  have to show that (i) there is $z\geq 0$ such that $\mathbf
  E[T_z|-\log X_0=z] < \infty$ and (ii) $\mathbf P(T_z<\infty|-\log
  X_0 = x) = 1$ for all $x\geq 0$.

  Let $z$ be large enough such that
  $$p_z := p(1-e^{-z}) > \log(1/p).$$
  We define
  $$S_{(z, z+\log(1/p)]}:=\inf\{t:z<-\log X_t\leq z+\log(1/p)\}.$$
  Then, $\mathbf E[S_{(z, z+\log(1/p)]} | -\log X_0 = x] < \infty$
  for all $x\leq z$. Indeed, the probability for at least
  $z/\log(1/p)$ jumps in some small time interval of length
  $\varepsilon>0$ is positive. After the first such time interval we
  can be sure that $S_{(z, z+\log(1/p)]}$ has occurred. By finitness of
  first moments of geometric distributions,
  $\mathbf E[S_{(z, z+\log(1/p)]} | -\log X_0 = x] < \infty$ follows.
  By a restart argument, we have to show that
  $\mathbf E[T_z | -\log X_0 = x] < \infty$ for all
  $z<x\leq z + \log(1/p)$, which will be done by using a comparison
  argument. For this, let $\mathcal R = (R_t)_{t\geq 0}$ be a process
  with generator
  $$ G_{\mathcal R}g(y) = -p_z g'(y) + g(y + \log(1/p)) - g(y).$$
  If $z < R_0 = -\log X_0 \leq z + \log(1/p)$, then -- using the same
  Poisson processes for $-\log\mathcal X$ and $\mathcal R$ -- we have
  that $T_z \leq T_z^{\mathcal R} := \inf\{t\geq 0:
  R_t = z\}$ since $p(1-e^{-y}) \geq p_z$ for $y\geq z$. Since $(R_t -
  R_0 + t(p_z - \log(1/p)))_{t\geq 0}$ is a martingale and
  $T_z^{\mathcal R}<\infty$ almost surely, we have by optional
  stopping that $\mathbf E[R_0 - R_{T^{\mathcal R}_z}] = R_0 - z =
  (p_z - \log(1/p)) \mathbf E[T_z^{\mathcal R}]$, hence $\mathbf
  E[T_z | -\log X_0 = x] \leq \log(1/p)/(p_z - \log(1/p)) < \infty$.
  It is now straight-forward to obtain the
  properties (i) and (ii) and we see that $-\log\mathcal X$ is
  ergodic. In particular, $-\log X_\infty < \infty$, i.e.\ $X_\infty >
  0$ almost surely.

  By the ergodic Theorem, we have that $\tfrac 1t \int_0^t 1-X_s ds
  \xrightarrow{t\to\infty} 1 - \mathbf E[X_\infty]$. This can be used
  when we study the martingale $\big(-\log X_t + \log X_0 - \int_0^t
  \log(1/p) - p(1-X_s)ds\big)_{t\geq 0}$. By dividing by $t$ and
  ergodicity, we see that
  \begin{align*}
    0 & = \lim_{t\to\infty} \frac 1t \log (X_0/X_t) - \frac 1t
    \int_0^t \log(1/p) - p(1-X_s)ds \\ & = - \log(1/p) + p(1-\mathbf
    E[X_\infty]), \intertext{i.e.}  \mathbf E[X_\infty] & = 1 - \tfrac
    1p \log(1/p).
  \end{align*}
  Now, since $\mathbf E[G_{-\log\mathcal X}f(X_\infty)]=0$, we find
  that for $f(x)=e^{-kx}$
  \begin{align*} 
    0 & = \mathbf E[kp (1-X_\infty)X^k_\infty + X^k_\infty(p^k -1)] \\ & =
    -pk \mathbf E[X_\infty^{k+1}] + (pk + p^k -1)\mathbf E[X^k_\infty]
    \intertext{or} \mathbf E[X_\infty^{k+1}] & = \frac{pk + p^k-1}{pk}
    \mathbf E[X_\infty^k] = \Big(1 - \frac{1-p^k}{pk}\Big) \mathbf
    E[X_\infty^k].
  \end{align*}
  By induction, we see that
  \begin{align*}
    \mathbf E[X_\infty^k] & = \Big(1 - \tfrac 1p \log\big(\tfrac
    1p\big)\Big) \cdot \prod_{\ell=1}^{k-1} \Big( 1 -
    \frac{1-p^\ell}{p\ell}\Big).
  \end{align*}
  Last, we consider the case $p=p^\ast$. Let $X_t^{(p)}$ be the Markov
  process with generator \eqref{eq:genX} for a specific value of
  $p$. If $p\mapsto X_0^{(p)}$ is constant, we can couple these processes
  by using the same jump times such that $X_t^{(p)} \leq X_t^{(p')}$ for
  $p<p'$. Therefore, 
  \begin{align*}
    0 & \leq \mathbf E[\limsup_{t\to\infty} X^{(p^\ast)}_t] \leq
    \inf_{p > p^\ast} \mathbf E[\limsup_{t\to\infty} X^{(p)}_t] = \inf_{p
      > p^\ast} \mathbf E[X^{(p)}_\infty] \\ & = \inf_{p > p^\ast}
    \Big(1 - \tfrac 1p\log\Big(\tfrac 1p\Big)\Big) = 0.
  \end{align*}
  Hence, $\limsup_{t\to\infty} X^{(p^\ast)}_t = \lim_{t\to\infty}
  X^{(p^\ast)}_t = 0$, almost surely.
\end{proof}

\subsection{Martingales, the Gamma function and a recursion}
We prepare some facts needed in the proofs of Theorems~\ref{T2}
and~\ref{T3}.

\begin{lemma}[Asymptotics for the Gamma function\label{l:gamma}]
  Let $n_0\geq 0$ and $a>-t_0$. Then,
  $$
    \prod_{k=n_0}^{n-1}\frac{k+a}{k}
      = \frac{\Gamma(n+a)}{\Gamma(n)}\cdot\frac{\Gamma(n_0)}{\Gamma(n_0+a)}
      \stackrel{n\to\infty}\sim \frac{n^a\Gamma(n_0)}{\Gamma(n_0+a)}.$$
\end{lemma}

\begin{proof}
  The first identity follows by iterating the functional equation
  $x\Gamma(x) = \Gamma(x+1)$ and for the asymptotics see e.g.\ \cite{abramowitz},
  6.1.46.
\end{proof}

\begin{lemma}[Martingale estimates\label{l:konvergenz}]
  \sloppy Let $\mathcal X = (X_n)_{n=n_1,n_0+1,...}$ be a
  non-negative, integrable stochastic process, adapted to a filtration
  $\mathcal F := (\mathcal{F}_n)_{n=n_0, n_0+1,...}$,
  $\mathcal{F}_\infty :=\sigma\Big(\bigcup\limits_{n = n_0}^\infty
  \mathcal{F}_n\Big)$ and $x_0:=\mathbf{E}[X_{n_0}]>0$. Moreover, let
  $a>-n_0$ and assume that
  $$ \mathbf{E}\left[X_{n+1}| \mathcal{F}_n\right] = \Big(1+\frac an\Big)X_n$$
  for all $n=n_0, n_0+1,...$ Then, the following holds:
  \begin{asparaenum}
  \item The process $\mathcal M=(M_n)_{n\geq n_0}$ defined by $M_{n_0} = X_{n_0}$
    and
    \begin{align*}
      M_n = X_n\cdot\prod_{k=n_0}^{n-1} \frac{k}{k+a}
    \end{align*}
    is an $\mathcal F$-martingale and the expectations of the $X_n$ hold
    \begin{align}
      \label{eq:mar1}
      \mathbf E[X_n] = x_0\cdot\prod_{k=n_0}^{n-1} \frac{k+a}{k}
      \stackrel{n\to\infty}\sim
      \frac{x_0\Gamma(n_0)}{\Gamma(n_0+a)}\cdot n^a.
    \end{align}
  \item There is a non-negative random variable $X_\infty
    \in\mathcal{L}^1(\mathcal{F}_\infty)$ with
    $\mathbf{E}[X_\infty]\leq x_0 \Gamma(n_0)/\Gamma(n_0+a)$ such that
    $$n^{-a}X_n
    \xrightarrow{n\to\infty} X_\infty \text{ almost surely.}$$ 
  \item If, in addition to 2., $\mathbf{E}[X_n^r]=O(n^{ar})$ for some
    $r>1$, then the convergence also holds in $\mathcal L^r$ and thus
    $\mathbf P(X_\infty>0)>0$.
  \end{asparaenum}
\end{lemma}

\begin{proof}
  By the assumptions and its definition it is easy to see that
  $\mathcal M$ is a martingale. Thus, the equality in \eqref{eq:mar1}
  follows by induction and the asymptotic expansion is a consequence
  of Lemma~\ref{l:gamma}.

  Since $\mathcal M$ is non-negative, it converges almost surely to
  some random variable $M_\infty\in\mathcal L^1(\mathcal F_\infty)$
  with $\mathbf E[M_\infty]\leq\mathbf E[M_{n_0}]$ and hence,
  $$ n^{-a} X_n = n^{-a} M_n \cdot \prod_{k=n_0}^{n-1}\frac{k+a}{k} 
  \xrightarrow{n\to\infty}M_\infty \cdot
  \frac{\Gamma(n_0)}{\Gamma(n_0+a)} =: X_\infty$$ almost surely. The
  convergence is also in $\mathcal L^r$ if $\mathcal M$ is $\mathcal
  L^r$-bounded. We compute, using $\mathbf E[X_n^r] \leq cn^{ar}$
  and Lemma~\ref{l:gamma},
  \begin{align*}
    \sup_{n}\mathbf{E}[M_n^r] =
    \sup_{n}\mathbf{E}[X_n^r]\Big(\prod\limits_{k=n_0}^{n-1}\frac{k}{k+a}\Big)^r
    \leq
    \sup_{n}\frac{\Gamma(n_0+a)^rcn^{ar}}{\Gamma(n_0)^rc'n^{ar}} <\infty,
  \end{align*}
  which shows the assertion. In particular, $\mathcal M$ converges in $\mathcal L^1$
  which gives us $\mathbf E[M_\infty]>0$ concluding the proof.
\end{proof}

\begin{lemma}[Recursions]\label{l:recursions}
  Let $n_0>0,a>-n_0,\varepsilon>0$ and $f,g:\{n_0,n_0+1,\ldots\}\to(0,\infty)$,
  satisfying
  \begin{align}
    f(n+1)=\Big(1+\frac an\Big)f(n)+\frac {g(n)}n \label{eq:l-rec1}
  \end{align}
  for all $n\geq n_0$. Then, for $n\to\infty$\\
  \begin{asparaenum}
    \item If $g=O(n^{a-\varepsilon})$, then $f=\Theta(n^a)$.\\[-0.5em]
    \item If $g=O(n^a)$, then $f=O(n^a\log n)$, and
          if $g=\Omega(n^a)$, then $f=\Omega(n^a\log n)$.\\[-0.5em]
    \item If $g=O(n^{a+\varepsilon})$, then $f=O(n^{a+\varepsilon})$, and
          if $g=\Omega(n^{a+\varepsilon})$, then $f=\Omega(n^{a+\varepsilon})$.\\[-0.5em]
  \end{asparaenum}
\end{lemma}
\begin{proof}
  At first note that from Lemma \ref{l:gamma} and the positivity of $g$
  we easily obtain $f=\Omega(n^a)$ in any case. Iteration of \eqref{eq:l-rec1}
  gives us
  \begin{align}\notag
    f(n)
      &= f(n_0)\prod_{k=n_0}^{n-1}\Big(1+\frac ak\Big)
        + \sum_{k=n_0}^{n-1}\frac{g(k)}k\prod_{\ell=k+1}^{n-1}\Big(1+\frac a\ell\Big)\\
      \label{eq:l-rec2}
      &=\prod_{k=n_0}^{n-1}\frac{k+a}k\cdot
        \Bigg(
          f(n_0)
        + \sum_{k=n_0}^{n-1}\frac{g(k)}k\prod_{m=n_0}^{k}\frac m{m+a}
        \Bigg).
  \end{align}
  Lemma \ref{l:gamma} provides constants $c_0,c_1,c_2>0$ which hold
  \begin{align}
  \eqref{eq:l-rec2}
    \leq c_0n^a\Big(f(n_0)+\sum_{k=n_0}^{n-1}\frac{g(k)}kc_1k^{-a}\Big)
    \leq c_2n^a\sum_{k=n_0}^{n-1}\frac{g(k)}{k^{a+1}}.\label{ineq:l-rec}
  \end{align}
  Now 1. and the first parts of 2. and 3. follow immediately by considering
  a suitable integral as upper bound for the sum. Lastly, note that
  Lemma \ref{l:gamma} also provides constants $c_0,c_1,c_2$ which satisfy the
  respective lower bounds in \eqref{ineq:l-rec}, such that the remaining
  claims follow analogously.
\end{proof}
\begin{corollary}\label{co:rec}
  If, in Lemma \ref{l:recursions}, $g$ is $O(n^b)$, then
  $f=o\big(n^{\max\{a,b\}+\varepsilon}\big)$
  for all $\varepsilon>0$.
\end{corollary}

\section{Proof of Theorem~\ref{T1}}
\label{S:proofT1}
\subsection{A time-continuous partial duplication graph}
It will be helpful to have a time-continuous version of $\mathcal G$. 

\begin{definition}[Partial duplication random graph
  PDt\label{def:PDt}]
  Let $p \in [0,1]$. We define the following random graph process --
  called \emph{time-continuous partial duplication random graph} or
  PDt graph -- $\mathcal G = (G_t)_t\geq 0$ with $G_t = (V_t, E_t)$,
  where $G_t$ is the graph at time $t$ with vertex set $V_t$ and
  (undirected) edge set $E_t \subseteq \{ \{v,w\}: v,w \in V_t, v\neq
  w\}$. Starting in some $G_0 = (V_0, E_0)$, every $v\in V_t$ gives
  rise at rate~$1 + 1/|V_t|$ to a duplication event. Upon such an
  event, a new node $v' \notin V_{t-}$ is created and every edge
  connected to~$v$ (i.e.\ every $e\in E_{t-}$ with $e = \{v,w\}$ for
  some $w\in V_{t-}$) is copied at time~$t$ with probability~$p$,
  i.e.\ $\{v',w\} \in E_t$ with probability~$p$, independently of all
  other edges.

  We define as in Definition~\ref{def:PDn} the degree distribution
  $F_k(t) := F_k(G_t)$ and $F_k^\circ(t) := F_k^\circ(G_t)$ and its
  probability generating function $H_q(t):=H_q(G_t)$ and $H_q^\circ(t) :=
  H_q^\circ(G_t)$.
\end{definition}

\begin{remark}[Connection between PDn and PDt]
  \begin{asparaenum}
  \item \sloppy We abuse notation here and use $(G_t)_{t\geq 0}$ for
    the time-continuous PDt graph while $(G_n)_{n=n_0,n_0+1,...}$ is
    the time-discrete PDn graph. Of course, these two processes are
    closely connected. Let $\tau_n := \inf\{t\geq 0: |V_t|=n\}$ Then,
    $(G_{\tau_n})_{n=n_0,n_0+1,...} \sim (G_{n})_{n=n_0,n_0+1,...}$
  \item The choice of the rate $1 + 1/|V_t|$ for initiating a
    duplication event seems unnatural. It will however turn out that
    this choice simplifies our line of argument; see the next
    proposition.
  \end{asparaenum}
\end{remark}

\noindent
We now derive an important relationship for $H^\circ_q(t)$.

\begin{proposition}[Evolution of $H_q^\circ(t)$\label{P:Hqt}]
  For $\mathcal G = (G_t)_{t\geq 0}$ and $H^\circ(t)$ as above,
  $$\frac{d}{dt} \mathbf E[H^\circ_q(t)]
   = \mathbf E\Big[-pq(1-q)\frac{d}{ds} H^\circ_s(t)\Big|_{s=q}
     + H^\circ_{1-p+pq}(t) - H^\circ_q(t)\Big].$$
\end{proposition}

\begin{remark}
  Later, it will be useful to define $x:=1-q$ and $\widetilde H_x(t)
  := H_{q}(t)$ in order to obtain
  \begin{align}\label{eq:wideH}
    \frac{d}{dt} \mathbf E[\widetilde H^\circ_x(t)] = \mathbf
    E\Big[px(1-x)\frac{d}{dx} \widetilde H^\circ_x(t) + \widetilde
    H^\circ_{px}(t) - \widetilde H^\circ_x(t)\Big].
  \end{align}
  In particular, note that the right hand side is reminiscent of
  \eqref{eq:genX}.
\end{remark}

\begin{proof}
  We have already seen the evolution of $n\mapsto\mathbf E[H_q(n)]$ in
  Proposition~\ref{P:FHSC}. From this, we derive, since the total rate
  for a duplication event at time $t$ is $|V_t|+1$,
  \begin{align*}
    \mathbf E[H_q(t+dt)] & = \mathbf E\Big[H_q(t)(1-(|V_t|+1)dt) \\ &
    \qquad + dt \cdot (|V_t|+1) \Big( H_q(t) -
    pq(1-q)\frac{d}{ds}H^\circ_s(t)\Big|_{s=q} + H^\circ_{1-q+pq}(t)\Big)\Big].
  \end{align*}
  From this, we obtain
  \begin{align*}
    \mathbf E[H^\circ_q(t+dt)] & = \mathbf
    E\Big[H^\circ_q(t)(1-(|V_t|+1)\cdot dt) \\ & \qquad + dt \cdot
    (|V_t|+1) \Big(\frac{|V_t|}{|V_t|+1} H^\circ_q(t) \\ & \qquad
    \qquad - \frac{1}{|V_t|+1} pq(1-q)\frac{d}{ds}H^\circ_s(t)\Big|_{s=q}
    + \frac{1}{|V_t|+1} H^\circ_{1-q+pq}(t)\Big)\Big] \\
    & = \mathbf E\Big[H^\circ_q(t) + dt \cdot \Big( -
    pq(1-q)\frac{d}{ds}H^\circ_s(t)\Big|_{s=q} +H^\circ_{1-q+pq} -
    H^\circ_q(t)\Big)\Big].
  \end{align*}
\end{proof}

\subsection{A duality relationship between $\mathcal X$ and $\mathcal G$}
\label{subs:duality}
Now, we make clear why we need the auxiliary process $\mathcal X$ from
Subsection~\ref{ss:mathcalX}. Here, we borrow ideas from the notion
of duality of Markov processes; see Chapter 4.4 in \cite{EK86}.

Recall that two Markov processes $\mathcal G = (G_t)_{t\geq 0}$ (which
will be the PDt-graph below) and $\mathcal X = (X_t)_{t\geq 0}$ (which
will be the piecewise-deterministic process from
Subsection~\ref{ss:mathcalX}) with state spaces $E$ and $E'$ are
called dual with respect to the function $H: E\times E' \to\mathbb R$
if
\begin{align}\label{eq:dual}
  \mathbf E[H(G_t, x) | G_0=g] = \mathbf E[H(g, X_t)|X_0=x]
\end{align}
for all $g\in E, x\in E'$. (In our application, $H$ will be the moment
generating function of the degree distribution of the PDt-graph
evaluated at $1-x$.) When one is interested in the process
$\mathcal G$, this relationship is most helpful if the process
$\mathcal X$ is easier to analyse than the process $\mathcal G$.
Moreover, frequently, the set of functions $\{H(.,x): x\in E'\}$ is
separating on $E$ such that the left hand side of~\eqref{eq:dual}
determines the distribution of $G_t$. In this case, the distribution
of the simpler process $\mathcal X$ determines via~\eqref{eq:dual} the
distribution of $\mathcal G$, so analysing $\mathcal G$ becomes
feasible. (In our application, however, $\{H(.,x): x\in E'\}$ is only
separating on the space of degree distributions and hence
\eqref{eq:dual} will determine the degree distribution of the
PDt-graph.)

There is no straight-forward way how to find dual processes, but they
arise frequently in the literature; see \cite{JansenKurt2014} for a
survey. Examples span reflected and absorbed Brownian motion,
interacting particle models such as the voter model and the contact
process, as well as branching processes.

\begin{proposition}[Duality\label{P:dual}]
  Let $\mathcal X = (X_t)_{t\geq 0}$ be a Markov process with state
  space $[0,1]$ with generator as given in \eqref{eq:genX} and
  $\widetilde H^\circ_x(t) := \sum_{k=0}^\infty F^\circ_k(t) (1-x)^k$
  as above. Then,
  \begin{align*}
    \mathbf E[\widetilde H^\circ_x(t)|G_0] = \mathbf E[\widetilde
    H^\circ_{X_t}(0)|X_0=x].
  \end{align*}
\end{proposition}

\begin{proof}
  On a probability space where $\mathcal X$ and the PDt-graph are
  independent, combining~\eqref{eq:wideH} and \eqref{eq:genX},
  \begin{align*}
    \frac{d}{ds}\mathbf E[\widetilde H^\circ_{X_{t-s}}(s)|G_0, X_0=x]
    = 0.
  \end{align*}
  The result then follows since $s\mapsto \mathbf E[\widetilde
  H^\circ_{X_{t-s}}(s)]$ is constant.
\end{proof}

\subsection{Proof of Theorem~\ref{T1}}
We start with the case $p\leq p^\ast$. Here, we know from
Lemma~\ref{l:X} that $X_t \xrightarrow{t\to\infty} 0$ almost
surely. Hence, using Proposition~\ref{P:dual}, for $q\in[0,1)$ and
$x:=1-q$,
\begin{align*}
  \lim_{n\to\infty}\mathbf E[H_q^\circ(n)] & = \lim_{t\to\infty}
  \mathbf E[H_q^\circ(t)] = \lim_{t\to\infty} \mathbf E[\widetilde
  H^\circ_{x}(t)] = \lim_{t\to\infty} \mathbf E[\widetilde
  H^\circ_{X_t}(0)|X_0=x] \\ & = \mathbf E[\widetilde H_0^\circ(0)] =
  \mathbf E[H_1^\circ(0)] = 1.
\end{align*}
In particular, since by~\eqref{P:FHSC}
\begin{align*}
  \mathbf E[H_0^\circ(n+1)|\mathcal F_n] &= 
                                           \mathbf E[F_0^\circ(n+1)|\mathcal F_n]
                                           = \frac n{n+1}F_0^\circ(n)+\frac1{n+1}\sum_{\ell\geq0}F_\ell^\circ(n)(1-p)^\ell
  \\ & \geq F_0^\circ(n)
       = H_0^\circ(n),
\end{align*}
$(H_0^\circ(n))_{n=0,1,2,...}$ is a bounded sub-martingale and thus
converges almost surely to 1.  By the monotonicity of the probability
generating function, we also obtain the stated uniform convergence
result.

The case $p>p^\ast$ can be treated similarly, but $X_t$ does not converge
almost surely to a constant. Hence, in this case with $X_\infty$ from
Lemma~\ref{l:X} we can compute
\begin{align*}
  \lim_{n\to\infty}\mathbf E[H_q^\circ(n)] &= \lim_{t\to\infty}\mathbf
  E[\widetilde H^\circ_{X_t}(0)|X_0=x]
  = \sum_{k=0}^\infty F^\circ_k(n_0)\mathbf E[(1-X_\infty)^k]\\
  &= \sum_{k=0}^\infty F^\circ_k(n_0)
  \sum_{\ell=0}^k \binom k \ell(-1)^\ell \mathbf E[X_\infty^\ell] \\
  &= \sum_{\ell=0}^\infty(-1)^\ell\mathbf E[X_\infty^\ell]
  \sum_{k=\ell}^\infty \binom k\ell F_k^\circ(n_0)\\
  &= 1 - \sum_{\ell=1}^\infty\frac{S_\ell^\circ(n_0)}{\ell!}
  (-1)^{\ell-1}\mathbf E[X_\infty^\ell].
\end{align*}
Now by Lemma~\ref{l:X}, the result follows.

\section{Proof of Theorems~\ref{T2} and \ref{T3}}
\label{S:proofT23}
\begin{proof}[Proof of Theorem~\ref{T2}]
  We start with 1. where we make use of Proposition~\ref{P:FHSC} and
  Lemma~\ref{l:konvergenz}. For the almost sure convergence in~\eqref{eq:T21},
  we use Lemma~\ref{l:konvergenz}.2 with $a=kp^{k-1}$, and for~\eqref{eq:T22},
  we use \eqref{eq:mar1}. We will set aside the $\mathcal L^2$ convergence for now.\\
  The proof of 2.\ is a bit more involved since the recursions from
  Proposition~\ref{P:FHSC} for $S_k$ involve both, $S_k$ and
  $S_{k-1}$. But considering the quantity
  $$Q_k(n) := \sum_{\ell=1}^ka_\ell S_\ell(n)$$ where, recalling that the empty product is~1,
  $$a_\ell:=\prod_{m=\ell}^{k-1}\frac{m(m+1)}{k-m+p^{k-1}-p^{m-1}},$$
  we obtain a fitting recursion as follows:
  \begin{align*}
    \mathbf E[Q_k(& n+1)|\mathcal F_n]
       =\sum_{\ell=1}^ka_\ell\Bigg(\Big(1+\frac{p\ell+p^\ell}n\Big)S_\ell(n)
                                     +\frac{p\ell(\ell-1)}nS_{\ell-1}(n)\Bigg)\\
      &=\Big(1+\frac{pk+p^k}n\Big)S_k(n)
        +\sum_{\ell=1}^{k-1}S_\ell(n)\Bigg(\Big(1+\frac{p\ell+p^\ell}n\Big)a_\ell
                                     +\frac{p\ell(\ell+1)}na_{\ell+1}\Bigg)\\
      &=\Big(1+\frac{pk+p^k}n\Big)S_k(n)
        +\sum_{\ell=1}^{k-1}a_\ell S_\ell(n)\Bigg(1+\frac{p\ell+p^\ell}n
                    +\frac{pk-p\ell+p^k-p^\ell}n\Bigg)\\
      &=\Big(1+\frac{pk+p^k}n\Big)Q_k(n).
  \end{align*}
  \sloppy Thus, by Lemma~\ref{l:konvergenz}.2 there are random
  variables $S_k(\infty)$ satisfying
  $n^{-(kp+p^k)}Q_k(n)\xrightarrow{n\to\infty}S_k(\infty)$ almost
  surely. Since $n^{\ell p+p^\ell}=o\big(n^{kp+p^k}\big)$ for all
  $\ell<k$ and~\eqref{eq:T22} provides the asymptotics of
  $S_1(n)=2C_2(n)$, inductively the almost sure convergence in
  \eqref{eq:T21b} follows.

  Now, writing $Q_1(n) = S_1(n)$ as well as $Q_2(n) = S_2(n) +
  \frac{2}{p}S_1(n)$, we have from Lemma~\ref{l:konvergenz}
  \begin{align*}
    \mathbf E[S_2(n)] & = \mathbf E[Q_2(n)] - \frac{2}{p}\mathbf
    E[Q_1(n)] \\ & = Q_2(n_0) \prod_{k=n_0}^{n-1} \frac{k + 2p + p^2}{k} -
    \frac{2}{p}Q_1(n_0)\prod_{k=n_0}^{n-1} \frac{k + 2p}{k}
  \end{align*}
  and \eqref{eq:T22b} follows.

  For the $\mathcal L^2$ convergece in~\eqref{eq:T21} first consider the number
  of pairs of $k$-cliques at time $n$, $\tfrac 12C_k(n)_{\downarrow 2}$.
  Now let $C_{k,\ell}(n)$ be the number of $\ell$-\emph{pairs} at time $n$, that is
  pairs of $k$-cliques which share exactly $\ell$ nodes. (e.g. two disjoint cliques
  form a 0-pair and a $(k-1)$-pair of $k$-cliques is a $(k+1)$-clique with an edge
  missing.) Thus, we obtain
  \begin{align}
    \tfrac 12C_k(n)_{\downarrow 2}
      = {C_k(n)\choose 2}
      = \sum_{\ell=0}^{k-1}C_{k,\ell}(n).\label{sum-pairs}
  \end{align}
  Supposing there is an $\ell$-pair of $k$-cliques at time $n$, there are four
  ways for new pairs to arise during the next time step:
\begin{enumerate}[1)]
  \item
    First of all, every new clique forms a $(k-1)$-pair with the clique it
    was duplicated from, since they only do not share the new node.
    As \eqref{eq:P14} shows, this happens $\tfrac{kp^{k-1}}nC_k(n)$ times
    on average during the next time step. In the next 3 cases we will
    ignore those events.\\
  \item
    One of the $2(k-\ell)$ not-shared nodes is chosen and the one clique of the
    pair it is contained in is duplicated. Then, since the new clique retains the
    $\ell$ nodes which are part of the non-duplicated clique of the pair, a new
    $\ell$-pair is formed.\\[0.5em]
    Corresponding probability: $\tfrac{2(k-\ell)}np^{k-1}$\\
  \item
    One of the $\ell$ shared nodes is chosen and both cliques of the pair
    are duplicated. Obviously, this way a new $\ell$-pair arises.
    Additionally the other two new pairs (one original and the copy of
    the other original respectively) are $(\ell-1)$-pairs, since those
    cliques do not share the new node.\\[0.5em]
    Corresponding probability: $\tfrac\ell np^{2k-\ell-1}$\\
  \item
    One of the $\ell$ shared nodes is chosen, but only one of the cliques
    is duplicated. Similarly to 3), a new $(\ell-1)$-pair arises.
    (Since the duplication of one clique fails, so does the creation
    of the new $\ell$-pair and one of the $(\ell-1)$-pairs.)\\[0.5em]
    Corresponding probability: $\tfrac\ell n\cdot 2p^{k-1}(1-p^{k-\ell})
                                =\tfrac\ell n 2p^{k-1}-\tfrac\ell n2p^{2k-\ell-1}$
\end{enumerate}
Following this, for $\ell\leq k-2$ we obtain
\begin{align}
  \mathbf E[&C_{k,\ell}(n+1)-C_{k,\ell}(n)\mid\mathcal F_n]\notag\\[0.5em]
   &= \frac{2(k-\ell)p^{k-1}+\ell p^{2k-\ell-1}}n C_{k,\ell}(n)
    + \frac{2(\ell+1) p^{k-1}}n C_{k,\ell+1}(n)\label{rec-pairs1}
  \intertext{and}
  \mathbf E[&C_{k,k-1}(n+1)-C_{k,k-1}(n)\mid\mathcal F_n]\notag\\[0.5em]
   &= \frac{2p^{k-1}+(k-1)p^k}n C_{k,k-1}(n)
    + \frac{kp^{k-1}}nC_k(n).\label{rec-pairs2}
\end{align}
Using \eqref{sum-pairs} we compute
\begin{align*}
  \mathbf E[&C_k(n+1)_{\downarrow 2}-C_k(n)_{\downarrow 2}\mid\mathcal F_n]
    = 2\sum_{\ell=0}^{k-1}\mathbf E[C_{k,\ell}(n+1)-C_{k,\ell}(n)\mid\mathcal F_n]\\
   &= \sum_{\ell=0}^{k-1}
        \frac{2(k-\ell)p^{k-1}+\ell p^{2k-\ell-1}}n 2C_{k,\ell}(n)
    + \sum_{\ell=1}^{k-1}
        \frac{2\ell p^{k-1}}n 2C_{k,\ell}(n)
    + \frac{2kp^{k-1}}nC_k(n)\\
   &= \frac{2kp^{k-1}}{n}C_k(n)_{\downarrow 2}
    + \frac 2n\sum_{\ell=0}^{k-1}\ell p^{2k-\ell-1}C_{k,\ell}(n)
    + \frac{2kp^{k-1}}nC_k(n)\\
   &= \frac{2kp^{k-1}}{n}C_k(n)^2
    + \frac{2p^k}n\sum_{\ell=1}^{k-1}\ell p^{k-1-\ell}C_{k,\ell}(n)
  \intertext{and thus}
  \mathbf E[&C_k(n+1)^2]\\
    &= \Big(1+\frac{2kp^{k-1}}{n}\Big)\mathbf E[C_k(n)^2]
     + \frac{2p^k}n\sum_{\ell=1}^{k-1}\ell p^{k-1-\ell}\mathbf E[C_{k,\ell}(n)]
     + \frac{kp^{k-1}}{n}\mathbf E[C_k(n)].
\end{align*}
Since $\mathbf E[C_k(n)]=O\big(n^{kp^{k-1}}\big)=O\big(n^{2kp^{k-1}-kp^{k-1}}\big)$,
for the use of Lemma \ref{l:recursions} it suffices to show the existence of
a $\delta>0$ holding $\sum_\ell\mathbf E[C_{k,\ell}(n)]=O\big(n^{2kp^{k-1}-\delta}\big)$.
From \eqref{rec-pairs2} it follows, that
\begin{align*}
  \mathbf E[C_{k,k-1}(n+1)]
   &= \Big(1+\frac{2p^{k-1}+(k-1)p^k}n\Big)\mathbf E[C_{k,k-1}(n)]
    + \frac 1n O\big(n^{kp^{k-1}}\big)\\
   &\overset{\ref{co:rec}}=O\Big(n^{\max\{2p^{k-1}+(k-1)p^k,kp^{k-1}\}+\varepsilon}\Big)
\end{align*}
for arbitrarily small $\varepsilon>0$. Using Corollary \ref{co:rec} again, inductively,
\eqref{rec-pairs1} implies
\[
  \mathbf E[C_{k,\ell}(n)]
    =O\Big(n^{\max\limits_{\ell\leq m\leq k}\big(2(k-m)p^{k-1}+m p^{2k-m-1}\big)+\bar\varepsilon}\Big)
\]
and hence
\begin{align}
  \sum_{\ell=1}^{k-1}\mathbf E[C_{k,\ell}(n)]
  =O\Big(n^{\max\limits_{1\leq m\leq k}\big(2(k-m)p^{k-1}+m p^{2k-m-1}\big)+\tilde\varepsilon}\Big)
  \label{cond:pairs-reclem}
\end{align}
for arbitrarily small $\tilde\varepsilon>0$. Since
\[
  \max\limits_{1\leq m\leq k}\Big(2(k-m)p^{k-1}+mp^{2k-m-1}\Big)
  \leq\max\limits_{1\leq m\leq k}(2k-m)p^{k-1}=2kp^{k-1}-p^{k-1},
\]
letting $\tilde\varepsilon=p^{k-1}/2$, \eqref{cond:pairs-reclem} satisfies
the conditions of Lemma \ref{l:recursions}.1, we finally obtain
$\mathbf E[C_k(n)^2]=O\big(n^{2kp^{k-1}}\big)$ and Lemma \ref{l:konvergenz}.3
applies.
\end{proof}

\begin{proof}[Proof of Theorem~\ref{T3}]
  For~\eqref{eq:T32}, we will show that
  \begin{align}\label{eq:T32c}
    \mathbf{P}(D_k(n)\leq \ell|D_k(n_0)=a) =
    \sum_{m=a}^\ell(-1)^{m-a}{\ell \choose m}{m-1\choose a-1}
    \prod_{j=n_0}^{n-1}\left(1-\frac{pm}{j}\right)
  \end{align}
  which implies~\eqref{eq:T32}.

  We fix $n_0, k$ and $a$ and set
  $$\Phi_\ell(n) := \mathbf P(D_k(n) \leq \ell| D_k(n_0)=a).$$
  We will prove \eqref{eq:T32c} by induction over $n$. For $n=n_0$, we
  have that $\Phi_\ell(n_0) = 1_{\ell\geq a}$. In addition, the right
  hand side of~\eqref{eq:T32c} gives for $n=n_0$
  \begin{align*}
    \sum_{m=a}^\ell(-1)^{m-a}{\ell \choose m}{m-1\choose a-1} & =
    \sum_{m}(-1)^{m-a}{(\ell-a)-(-a)\choose m}{-1 + m \choose -a+m} \\
    & = (-1)^{\ell-a}\binom{-1}{\ell-a} = \binom{\ell-a}{\ell-a} =
    1_{\ell\geq a}
  \end{align*}
  according to \cite{Riordan1968}, (8) and (ii) in Chapter~1. This
  shows
  that~\eqref{eq:T32c} holds for $n=n_0$ and all $\ell$. In order to
  apply induction, we get the recursion
  \begin{align*}
    \Phi_{\ell}(n+1) &= \Phi_{\ell}(n) - \frac{p\ell}{n} \mathbb{P}(D_k(n)=\ell | D_k(n_0)=a)\\
    &= \Phi_{\ell}(n) - \frac{p\ell}{n}\cdot \Big(\Phi_{\ell}(n)-
    \Phi_{\ell-1}(n)\Big)
  \end{align*}
  since $D_k$ increases by at most one in every time step.

  Assume that~\eqref{eq:T32c} holds for an $n$ for all $\ell$. Then,
  using the recursion, and the assumption for $n$,
  \begin{align*}
    \Phi_\ell(n+1) & = \sum_{m=a}^\ell(-1)^{m-a}{\ell \choose
      m}{m-1\choose a-1}
    \prod_{j=n_0}^{n-1}\left(1-\frac{pm}{j}\right) \\ & \qquad -
    \sum_{m=a}^{\ell}(-1)^{m-a}\frac{p\ell}{n} \Big({\ell \choose m} -
    {\ell-1 \choose m}\Big){m-1\choose a-1}
    \prod_{j=n_0}^{n-1}\left(1-\frac{pm}{j}\right) \\ & =
    \sum_{m=a}^\ell(-1)^{m-a}{\ell \choose m}{m-1\choose a-1}
    \prod_{j=n_0}^{n-1}\left(1-\frac{pm}{j}\right) \\ & \qquad \qquad
    - \frac{pm}{n} \cdot \sum_{m=a}^{\ell}(-1)^{m-a} {\ell \choose
      m}{m-1\choose a-1}
    \prod_{j=n_0}^{n-1}\left(1-\frac{pm}{j}\right) \\ & =
    \sum_{m=a}^\ell(-1)^{m-a}{\ell \choose m}{m-1\choose a-1}
    \prod_{j=n_0}^{n}\left(1-\frac{pm}{j}\right)
  \end{align*}
  and we are done.

  For \eqref{eq:T33}, we will use Lemma~\ref{l:konvergenz}. 
  We have that 
  \begin{align*}
    D_k(n+1) - D_k(n) = \begin{cases} 1, & \text{ with probability } p\frac{D_k(n)}{n},\\
      0, & \text{ with probability } 1 - p\frac{D_k(n)}{n} \end{cases}
  \end{align*}
  since $D_k$ increases by one iff one neighbor of $v_k$ and the
  respective edge are copied. Using Lemma~\ref{l:gamma} and that
  $D_k(n)\xrightarrow{n\to\infty}\infty$ we obtain for
  $r>-1\geq-D_k(n_0)$ that
  \begin{align*}
    D_k(n)^r\sim\frac{\Gamma(D_k(n)+r)}{\Gamma(D_k(n)}
  \end{align*}
  almost surely, where the right hand side satisfies
  \begin{align*}
    \mathbf E\Big[ & \frac{\Gamma(D_k(n+1)+r)}{\Gamma(D_k(n+1))}\Big|\mathcal F_n\Big]\\[1em]
     &= \frac{pD_k(n)}n\cdot\frac{\Gamma(D_k(n)+1+r)}{\Gamma(D_k(n)+1)}
       +\Big(1-\frac{pD_k(n)}n\Big)\cdot\frac{\Gamma(D_k(n)+r)}{\Gamma(D_k(n))}\\[1em]
     &= \frac{\Gamma(D_k(n)+r)}{\Gamma(D_k(n))}\cdot\Big(
         \frac{pD_k(n)}n\cdot\frac{D_k(n)+r}{D_k(n)}+1-\frac{pD_k(n)}n\Big)\\[1em]
     &= \frac{\Gamma(D_k(n)+r)}{\Gamma(D_k(n))}\cdot\Big(1+\frac{pr}n\Big).
  \end{align*}
  Thus, Lemma~\ref{l:konvergenz}.2 shows
  \begin{align*}
    n^{-rp}\frac{\Gamma(D_k(n)+r)}{\Gamma(D_k(n))}
     \sim \Big(n^{-p}D_k(n)\Big)^r
     \xrightarrow{n\to\infty}D_k(\infty)^r
  \end{align*}
  almost surely. Furthermore, Lemma~\ref{l:konvergenz}.1 gives us the
  $\mathcal L^r$-boundedness for $r>1$ we need for Lemma~\ref{l:konvergenz}.3.
  Hence, we obtain the $\mathcal L^r$-convergence of $n^{-p}D_k(n)$
  and~\eqref{eq:T34}. Lastly, Lemma~\ref{l:konvergenz}.1 also shows the
  convergence of $(n^{-p}D_k(n))^{-\frac12}$ to an integrable and thus
  finite random variable which delivers the almost sure positivity of
  $D_k(\infty)$.
\end{proof}



\end{document}